\documentclass[a4paper,12pt,reqno,intlimits]{article}
\usepackage[utf8]{inputenc}
\usepackage[T2A]{fontenc}
\usepackage[english,russian]{babel}
\usepackage{amsfonts,amsmath,amsxtra,amsthm,amssymb,latexsym,mathrsfs,graphicx,mathtools,cite,nicefrac,url}
\usepackage{algorithm}
\usepackage{algorithmic}
\usepackage{multirow}
\usepackage{authblk}

\theoremstyle{plain}
\newtheorem{theorem}{Теорема}

\newtheorem{definition}{Определение}

\theoremstyle{definition}
\newtheorem{remark}{Замечание}
\newtheorem{example}{Пример}
\newtheorem{task}{Задача}

\numberwithin{theorem}{section}
\numberwithin{lemma}{section}
\numberwithin{proposition}{section}
\numberwithin{corollary}{section}
\numberwithin{remark}{section}
\numberwithin{definition}{section}
\numberwithin{example}{section}

\floatname{algorithm}{Алгоритм}

\newcommand\br[1]{\left(#1\right)}

\newcommand\mbr[1]{\left|#1\right|}
\newcommand\abr[1]{\left\langle#1\right\rangle}
\newcommand\nbr[1]{\left\|#1\right\|}
\newcommand\cebr[1]{\left\lceil#1\right\rceil}

\title{АДАПТАЦИЯ К ВЕЛИЧИНАМ ПОГРЕШНОСТЕЙ ДЛЯ НЕКОТОРЫХ МЕТОДОВ ОПТИМИЗАЦИИ ГРАДИЕНТНОГО ТИПА\footnote{Работа выполнена при поддержке Российского научного фонда, проект 18-71-00048.}}
\author{Ф.\,С.\,Стонякин}
\date{}

\begin{document}

\maketitle

\section{Введение}

Хорошо известно, что методы градиентного типа отличаются относительной простотой и малыми затратами памяти, что объясняет  их популярность в работах по многомерной оптимизации (см., например, \cite{ Necoara_2019, s2, inexact_model_2019, s3, s1, DevolderThesis, s5, Nesterov_2013, Karimi}).
 Напомним, что для вывода оценок скорости сходимости градиентного метода можно использовать идею аппроксимации функции в исходной точке (текущем положении метода) мажорирующим ее параболоидом вращения. Так, для задачи минимизации выпуклого функционала $f:Q\rightarrow\mathbb{R}$, заданного на выпуклом замкнутом множестве $Q \subset \mathbb{R}^n$ с липшицевым градиентом (для некоторой константы $L>0$ при произвольных $x,y\in Q$ верно $\nbr{\nabla f(x)-\nabla f(y)}\leqslant L\nbr{x-y}$), выполняются неравенства
\begin{equation}\label{eq01}
f(x)+\abr{\nabla f(x),y-x}\leqslant f(y)\leqslant f(x)+\abr{\nabla f(x),y-x}+\frac{L}{2}\nbr{y-x}^2.
\end{equation}
Неравенства \eqref{eq01} позволяют получить для обычного градиентного спуска с постоянным шагом оценку скорости сходимости
\begin{equation}\label{eq02}
f(\hat{x})-f^*\leqslant\frac{C_1}{N},
\end{equation}
где $\hat{x}$~--- выход работы метода после $N$ итераций, $f^*$~--- точное значение искомого минимума функции $f$, $C_1$~--- некоторая положительная константа.

В новых работах, посвященных методам градиентного типа, например в~\cite{s3}, введено условие относительной гладкости оптимизируемого функционала, предполагающее замену правого неравенства в \eqref{eq01} на ослабленный вариант
\begin{equation}\label{eq01new}
f(y)\leqslant f(x)+\abr{\nabla f(x),y-x}+LV(y, x),
\end{equation}
где $V(y, x)$ --- широко используемый в оптимизации аналог расстояния между точками $x$ и $y$, который называют \emph{расхождением Брэгмана}. Обычно \emph{расхождение Брэгмана} вводится с использованием вспомогательной 1-сильно выпуклой функции $d$ (порождает расстояния), которая дифференцируема во всех точках $x\in Q$,
\begin{equation}\label{3}
V(y,x)=d(y)-d(x)-\langle \nabla d(x),y-x\rangle \quad \forall x,y\in Q;
\end{equation}
 здесь $\langle \cdot,\cdot\rangle$ --- скалярное произведение в $\mathbb{R}^n$, причем ввиду 1-сильной выпуклости функции~$d$  для произвольных $x, y \in Q$ верно неравенство $V(y,x) \geqslant 1/2\, \|y - x\|^2$. В частности, для стандартной евклидовой нормы $\|\cdot\|_2$ и соответствующего расстояния в $\mathbb{R}^n$ можно считать, что $V(y,x) = d(y-x) = 1/2\, \|y - x\|_2^2$ для произвольных $ x, y \in Q$. Однако рассмотренное в \cite{s3} условие относительной гладкости предполагает лишь выпуклость (но не сильную выпуклость) порождающей функции $d$. Как показано в \cite{s3}, концепция относительной гладкости позволяет применить вариант градиентного метода для некоторых задач, к которым ранее применялись лишь методы внутренней точки.

Весьма естественно возникает вопрос влияния на скорость сходимости метода погрешностей задания целевой функции и/или градиента. В этом плане можно отметить хорошо известную концепцию неточного оракула О.\,Деволдера --- Ф.\,Глинера --- Ю.\,Е.\,Нестерова \cite{s1, DevolderThesis}. Говорят, что функция $f$ допускает неточный оракул $(f_{\delta}(x),g_{\delta}(x))\in\mathbb{R} \times \mathbb{R}^n$ в произвольной запрошенной точке $x \in Q$, если для некоторых $\delta>0$ и $L >0$ выполняется аналог неравенства~\eqref{eq01}:
\begin{equation}\label{eq04}
f_{\delta}(x)+\abr{g_{\delta}(x),y-x}\leqslant f(y)\leqslant f_{\delta}(x)+\abr{g_{\delta}(x),y-x}+\frac{L}{2}\nbr{y-x}^2+\delta\,\,\,\forall x,y\in Q.
\end{equation}
По сути, \eqref{eq04} означает, что $f_{\delta}(x)$ --- приближенное значение $f(x)$ \mbox{($f_{\delta} (x) \leqslant f(x) \leqslant f_{\delta} (x) + \delta$)}, а $g_{\delta}(x)$~--- некоторый $\delta$-субградиент $f$ в произвольной точке $x$. Оказывается \cite{s1}, что при выполнении условия \eqref{eq04} для градиентного метода (с заменой пары ($f,\nabla f$) на ($f_{\delta},g_{\delta}$)) верна  оценка скорости сходимости
\begin{equation}\label{eq05}
f(\hat{x})-f^*\leqslant\frac{C_1}{N} + \delta,
\end{equation}
т.\,е. в оценке не происходит накопления величин, соответствующих погрешностям.

Идеология Деволдера --- Глинера --- Нестерова была развита в работе \cite{s2}, где обобщена концепция ($\delta, L$)-оракула и введено понятие ($\delta, L$)-модели целевой функции. Идея этого обобщения заключается в том, что линейная функция $\langle \nabla f(y), x - y \rangle$ в неравенстве \eqref{eq04} заменяется на некоторую абстрактную выпуклую функцию $\psi(x, y)$  \cite{s2}.
\begin{definition}
\label{gen_delta_L_oracle}
Говорят, что функция $f$ допускает $(\delta, L)$-модель $(f_{\delta}(x), \psi(y,x))$ в точке $x \in Q$, если для любого $y \in Q$ справедливо неравенство
\begin{gather}
\label{exitLDLOrig}
0 \leq f(y) - f_{\delta}(x) - \psi(y,x) \leq \frac{L}{2} \|y - x\|^2 + \delta,
\end{gather}
где $\psi_{\delta}(x,x)=0 \,\,\, \forall x \in Q$ и $\psi_{\delta}(x,y)$~--- выпуклая функция по $x$ для всякого $y \in Q$.
\end{definition}

Концепция из определения \ref{gen_delta_L_oracle} позволяет обосновать сходимость градиентного метода для достаточно широкого класса задач оптимизации \cite{s2, inexact_model_2019}. По сути, она дает возможность унифицировать подходы к различным на первый взгляд классам задач оптимизации с описанием степени влияния погрешностей данных на гарантированное качество решения, достижимое в ходе работы метода.

В настоящей работе предлагается модификация концепции ($\delta, L$)-модели целевой функции, которая учитывает возможность неточного задания не только значения целевой функции, но и самой функции-модели. В частности, для линейной модели $\psi(x, y) = \langle \nabla f(x), y - x \rangle$ описывается ситуация некоторой модификации условий \eqref{exitLDLOrig} с учетом отдельно погрешности задания $f$ и $\nabla f$. Если положить, что $\forall x\in Q$ справедливо
\begin{equation}\label{eq07}
\big\|\nabla f(x)-\widetilde{\nabla}f(x)\big\|\leqslant\Delta,\quad \Delta>0
\end{equation}
для некоторого доступного приближенного значения $\widetilde{\nabla}f(x)$ градиента $\nabla f$, то будет верно неравенство
$\big|\big\langle\nabla f(x)-\widetilde{\nabla}f(x),y-x\big\rangle\big|\leqslant\Delta\nbr{y-x},$
т.\,е. для всяких $x,y\in Q$
$$f(y)\leqslant f(x)+\big\langle\widetilde{\nabla}f(x),y-x\big\rangle+\frac{L}{2}\nbr{y-x}^2+\Delta\nbr{y-x},$$ а также $f(y)\geqslant f(x)+\big\langle\widetilde{\nabla}f(x),y-x\big\rangle-\Delta\nbr{y-x}.$

Если кроме этого учесть неравенство $f_{\delta}(x)\leqslant f(x)\leqslant f_{\delta}(x)+\delta$ при $\delta>0$,
то получим следующий аналог \eqref{eq04}:
$$
f_{\delta}(x)+\big\langle\widetilde{\nabla}f(x),y-x\big\rangle-\Delta\nbr{y-x}\leqslant f(y)$$
\begin{equation}\label{eq09}
\leqslant f_{\delta}(x)+\big\langle\widetilde{\nabla}f(x),y-x\big\rangle+\frac{L}{2}\nbr{y-x}^2+\Delta\nbr{y-x}+\delta\,\,\,\forall x,y\in Q,
\end{equation}
откуда $f_{\delta}(x) \leqslant f(x) \leqslant f_{\delta}(x) + \delta$.

В разд.~2 настоящей работы мы рассмотрим (в модельной общности подобно определению~\ref{gen_delta_L_oracle}) следующий аналог неравенства \eqref{eq09} c параметрами $\delta,\gamma, \Delta \geqslant 0$:
$$
f_{\delta}(x)+\big\langle\widetilde{\nabla}f(x),y-x\big\rangle-\gamma\nbr{y-x}\leqslant f(y)$$
\begin{equation}\label{eqconceptnew}
\leqslant f_{\delta}(x)+\big\langle\widetilde{\nabla}f(x),y-x\big\rangle+\frac{L}{2}\nbr{y-x}^2+\Delta\nbr{y-x}+\delta\,\,\,\forall x,y\in Q.
\end{equation}

Смысл такого обобщения заключается в возможности различных значений параметров $\gamma$ и $\Delta$ в \eqref{eqconceptnew}. Один из основных результатов работы --- обоснование потенциального уменьшения влияния $\Delta$ на оценку скорости сходимости метода. Отметим еще, что далее в разд.~3 подробно разобрано несколько примеров негладких задач, когда $\delta = \gamma = 0$ при $\Delta > 0$. Если положить $\gamma = 0$, то $\widetilde{\nabla}f(x)$ --- $\delta$-субградиент $f$ в точке $x$, и параметр $\Delta > 0$ может указывать в этом случае на скачки $\widetilde{\nabla}f(x)$ в точках негладкости $f$. Если положить $\delta = 0$, то при $\gamma > 0$ $\widetilde{\nabla}f(x)$ --- так называемый аналитический $\gamma$-субградиент $f$ \cite[Section~1.3]{Mordukh}. В итоге мы предлагаем максимально общую концепцию неточной модели целевой функции, которая могла бы охватить все указанные ситуации. Для функций, допускающих существование такой модели в любой запрошенной точке, мы предлагаем адаптивный градиентный метод (алгоритм~1) и доказываем теорему о скорости его сходимости (теорема \ref{th01}).

Неравенства \eqref{eq09} и \eqref{eqconceptnew} аналогичны \eqref{eq04}, но величины $\Delta\nbr{y-x}$ и $\gamma\nbr{y-x}$ уже зависят от выбора $x$ и $y$. Заменить их обе в \eqref{eqconceptnew} на постоянные величины, вообще говоря, возможно только в случае ограниченного допустимого множества задачи $Q$. Более того, хорошо известно, что при использовании $\widetilde{\nabla}f(x)$ из \eqref{eq07} метод может расходиться \cite[Section~4]{DevolderThesis}. Поэтому важно выделить класс задач, для которых можно получать приемлемые оценки скорости сходимости на неограниченных множествах. Это, в частности, мотивировало вторую часть основных результатов работы (разд.~4). Хорошо известно, что в случае сильной выпуклости целевого функционала оценки скорости сходимости градиентного метода могут существенно улучшаться. Например, для сильно выпуклого целевого функционала с липшицевым градиентом известно, что градиентный метод сходится с линейной скоростью. Весьма интересен и важен вопрос о том, насколько можно условие сильной выпуклости ослабить. В этом случае известен подход, основанный на использовании вместо сильной выпуклости условия градиентного доминирования Поляка --- Лоясиевича \cite{Polyak_1963} (см. также недавнюю работу \cite{Karimi} и имеющиеся там ссылки)
\begin{equation}\label{eq18}
f(x)-f(x_*)\leqslant\frac{1}{2\mu}\nbr{\nabla f(x)}^2 \quad \forall x \in Q,
\end{equation}
где $x_*$ --- точное решение задачи минимизации $f$, а $\mu>0$ --- некоторая постоянная и норма евклидова. Известно, что неравенство \eqref{eq18} в предположении липшицевости градиента с константой $L$ позволяет получить оценку скорости сходимости градиентного метода с постоянным шагом
\begin{equation}\label{1.16}
f(x^N)-f(x_*)\leqslant\br{1-\frac{\mu}{L}}^N\br{f(x^0)-f(x_*)}\leqslant\exp\br{-\frac{\mu}{L}N}\br{f(x^0)-f(x_*)}.
\end{equation}

В настоящей работе мы рассматриваем следующий ослабленный вариант условия $L$-липши-цевости градиента
\begin{equation}\label{PLSS1}
f(y)\leqslant f(x)+\big\langle\widetilde{\nabla}f(x),y-x\big\rangle + \frac{L}{2}\|y-x\|^{2}+\Delta\|y-x\|+\delta\quad \forall\,x,y\in Q
\end{equation}
для некоторых $\delta$ и $\Delta>0$. Например, это предположение естественно в случае, если значения функции $f$ немного отличаются от значений некоторой достаточно гладкой функции~$\widetilde{f}$, удовлетворяющей условию Липшица градиента (при этом $\widetilde{\nabla}f(x)$ --- некоторое возмущенное с точностью $\Delta$ значение градиента $\nabla f(x)$). По сути, в разд.~4  работы левая часть неравенства~\eqref{eqconceptnew} заменяется условием градиентного доминирования. Мы предлагаем метод с адаптивным подбором шага с настройкой на величины $L$, $\Delta$ и $\delta$ и показываем оценку скорости сходимости, аналогичную \eqref{1.16}. В частности, запуск предлагаемого метода (алгоритм~2) не предполагает знания никакой верхней оценки $L$ и может применяться для задач с неточным заданием градиента на неограниченных допустимых множествах. Более того, возможно использование данного подхода для некоторого класса негладких задач (см. определение \ref{Main_Def}).

Подытожим основные результаты (вклад) настоящей работы:

 --- В раз.~2 обобщено ранее предложенное в \cite{s2} понятие $(\delta,L)$-модели целевой функции в запрошенной точке и введена концепция ($\delta, \gamma, \Delta, L$)-модели функции (определение \ref{def1}). Предложен градиентный метод (алгоритм 1) для задач выпуклой минимизации с адаптивным выбором шага и адаптивной настройкой на некоторые из параметров ($\delta, \gamma, \Delta, L$)-модели, получена оценка качества решения в зависимости от номера итерации (теорема \ref{th01}).

--- В разд.~3 рассмотрен специальный класс задач выпуклой негладкой оптимизации, к которым применима концепция определения \ref{def1} ($\delta = \gamma = 0$, $\Delta>0$). Показано, что для таких задач возможно модифицировать алгоритм 1 так, чтобы гарантированно имела место сходимость по функции со скоростью $O(\varepsilon^{-2} \log_2 \varepsilon^{-1})$, которая близка к оптимальной на классе задач выпуклой негладкой оптимизации (теорема \ref{thm2ston}). При этом рассмотрены примеры негладких задач, для которых за счет адаптивности алгоритма~1 может наблюдаться существенно более высокая скорость сходимости.

--- В разд.~4 предложен адаптивный градиентный метод (алгоритм 2) для целевых функционалов с липшицевым градиентом (а также некоторой релаксацией этого условия), удовлетворяющих условию Поляка --- Лоясиевича. При этом учитывается возможность неточного задания градиента и предлагается адаптивная настройка работы метода на основные входные параметры. Обоснована линейная сходимость метода с точностью до величины, связанной с погрешностью (теорема \ref{thm3ston}).

\section{Концепция $(\delta,\gamma,\Delta,L)$-модели функции в запрошенной точке и оценка скорости сходимости для градиентного метода}\setcounter{equation}{0}

Введем анонсированный выше аналог понятия $(\delta,\gamma,\Delta,L)$-модели целевой функции, который учитывает погрешность $\Delta$ задания градиента и применим также для задач с относительно гладкими целевыми функционалами \cite{s3}.
\begin{definition}\label{def1}
Будем говорить, что $f$ допускает $(\delta,\gamma,\Delta,L)$-модель в точке $x \in Q$, если для некоторой выпуклой по первой переменной функции $\psi(y,x)$ такой, что $\psi(x,x)=0$ для произвольных $x,y\in Q$, будет верно неравенство
\begin{equation}\label{eq30}
\begin{split}
f_{\delta}(x)+\psi(y,x) - \gamma\nbr{y-x} \leqslant f(y) \leqslant f_{\delta}(x)+\psi(y,x)+\delta+\Delta\nbr{y-x}+LV(y,x).
\end{split}
\end{equation}
\end{definition}

Покажем пример, поясняющий смысл использования модельной общности в предыдущем определении.\smallskip

\begin{example}
Отметим задачу выпуклой композитной оптимизации $f(x)=g(x)+h(x)\rightarrow\min$, где $g$~--- гладкая выпуклая функция, а $h$~--- выпуклая необязательно гладкая функция простой структуры (операция проектирования на любое множество уровня $h$ несильно затратна). Если при этом для градиента $\nabla g$ задано его приближение $\widetilde{\nabla}g$:
$\big\|\widetilde{\nabla}g(x)-\nabla g(x)\big\|\leqslant\Delta,$ причем $g(y) \geqslant g(x) + \langle \widetilde{\nabla}g(x), y - x \rangle - \gamma \|y-x\|-\delta$, то можно положить $\psi(y,x) = \langle \widetilde{\nabla}g(x),y-x \rangle+h(y)-h(x)$,
и будет верно \eqref{eq30}. Композитная оптимизация весьма часто возникает во многих прикладных задачах (см. например \cite{s5}).\smallskip

\end{example}

Рассмотрим следующий метод для минимизации выпуклых функций, которые допускают существование $(\delta, \gamma, \Delta, L)$-модели во всякой точке $x \in Q$ и докажем результат о его скорости сходимости.\smallskip

\begin{algorithm}
\begin{algorithmic}\small
\caption{Адаптивный градиентный метод для выпуклых функций, допускающих $(\delta, \gamma, \Delta, L)$-модель в произвольной запрошенной точке.}
\REQUIRE $x^0$~--- начальная точка, $V(x_*,x^0)\leqslant R^2$,
параметры $\delta_0,\;L_0,\;\Delta_0$\\
($\delta_0\leqslant2\delta,\;L_0\leqslant2L,\;\Delta_0\leqslant2\Delta$).
\STATE $L_{k+1}:=\nicefrac{L_k}{2}$, $\Delta_{k+1}:=\nicefrac{\Delta_k}{2}$, $\delta_{k+1}:=\nicefrac{\delta_k}{2}$.
\STATE $x^{k+1}:=\text{arg}\min\limits_{x\in Q} \{ \psi(x,x^k)+L_{k+1}V(x,x^k) \}$.
\IF{$f_{\delta}(x^{k+1})\leqslant f_{\delta}(x^k)+\psi(x^{k+1},x^k)+L_{k+1}V(x^{k+1}, x^k)+\Delta_{k+1}\nbr{x^{k+1}-x^k}+\delta_{k+1}$}
\STATE $k:=k+1$ и выполнение п.~1.
\ELSE
\STATE $L_{k+1}:=2\cdot L_{k+1};\;\Delta_{k+1}:=2\cdot\Delta_{k+1};\;\delta_{k+1}:=2\cdot\delta_{k+1}$ и выполнение п.~2.
\ENDIF
\ENSURE $\widehat{x}:=\frac{1}{S_N}\sum\limits_{k=0}^{N-1}\frac{x^{k+1}}{L_{k+1}},\;S_N:=\sum\limits_{k=0}^{N-1}\frac{1}{L_{k+1}}$.
\end{algorithmic}
\end{algorithm}

Справедливо следующее утверждение.
\begin{theorem}\label{th01}
Пусть $f: Q \rightarrow\mathbb{R}$ --- выпуклая функция и для некоторой постоянной $R>0$ имеет место $V(x_*,x^0)\leqslant R^2$, где $x^0$~--- начальное приближение, а $x_*$~--- точное решение задачи минимизации $f$, ближайшее к $x^0$ с точки зрения расхождения Брэгмана. Тогда после $N$ итераций для выхода $\hat{x}$ алгоритма $1$ будет верно неравенство
\begin{equation}\label{eq10}
f(\hat{x})-f(x_*)\leqslant\frac{R^2}{S_N} +\frac{1}{S_N}\sum\limits_{k=0}^{N-1}\frac{\delta_k+\Delta_k\nbr{x^{k+1}-x^k}+\gamma \nbr{x^{k}-x_*}}{L^{k+1}} + \delta.
\end{equation}
При этом количество обращений к задаче п. $2$ листинга алгоритма $1$ не превышает
\begin{equation}\label{Oracle_Estim}
2N + \max\Big\{\log_2\frac{2L}{L_0}, \log_2\frac{2\delta}{\delta_0}, \log_2\frac{2\Delta}{\Delta_0}\Big\}.
\end{equation}
\end{theorem}
\begin{proof}
1. Согласно лемме 1 из \cite{s2} после завершения $k$-й итерации ($k = 0, 1, 2,\ldots$) алгоритма 1 будут верны неравенства
$$
\psi(x^{k+1},x^k) \leqslant \psi(x,x^k) + L_{k+1}V(x, x^k) - L_{k+1}V(x, x^{k+1}) - L_{k+1}V(x^{k+1}, x^k),
$$
$$
f_{\delta}(x^{k+1})\leqslant f_{\delta}(x^k)+\psi(x^{k+1},x^k)+L_{k+1}V(x, x^k)-L_{k+1}V(x, x^{k+1})+\Delta_{k+1}\big\|x^{k+1}-x^k\big\|+\delta_{k+1}.
$$
Поэтому ввиду того, что $f_{\delta}(x) \leq f(x) \leq f_{\delta}(x) + \delta$ при всяком $x \in Q$ имеем
$$
f(x^{k+1})\leqslant f(x^k)+\psi(x,x^k)+L_{k+1}V(x, x^k) - L_{k+1}V(x, x^{k+1})+\Delta_{k+1}\big\|x^{k+1}-x^k\big\|+\delta_{k+1} + \delta.
$$
Далее, с учетом левой части неравенства \eqref{eq30} при $x= x_*$ получим
$$
f(x^{k+1}) - f(x_*) \leqslant L_{k+1}V(x_*, x^k) - L_{k+1}V(x_*, x^{k+1})+\Delta_{k+1}\big\|x^{k+1}-x^k\big\|+\delta_{k+1} + \delta + \gamma \big\|x^{k}-x_*\big\|,
$$
откуда после суммирования по $k = 0, 1, \ldots, N-1$ ввиду выпуклости $f$ имеем
$$
f(\hat{x})-f(x_*)\leqslant \frac{1}{S_N} \sum\limits_{k = 0}^{N - 1} \frac{f(x^{k+1})}{L_{k+1}} - f(x_*) \leqslant V(x_*, x^0)
$$
$$
{ } + \frac{1}{S_N} \sum\limits_{k = 0}^{N - 1} L_{k+1}^{-1}\Big(\Delta_{k+1}\big\|x^{k+1}-x^k\big\|+\delta_{k+1}+\gamma \big\|x^{k}-x_*\big\|\Big) + \delta.
$$

2. Проверим оценку \eqref{Oracle_Estim}. Пусть на $(k+1)$-й итерации ($k = 0, 1, \ldots, N-1$) алгоритма 1 вспомогательная задача решается $i_{k+1}$ раз. Тогда
$$
2^{i_{k+1} - 2} = \frac{L_{k+1}}{L_k} = \frac{\delta_{k+1}}{\delta_k} = \frac{\Delta_{k+1}}{\Delta_k},
$$
поскольку в начале каждой итерации параметры $L_k, \delta_k, \Delta_k$ делятся на 2. Поэтому
$$
\sum\limits_{k = 0}^{N-1} i_{k+1} = 2N + \log_2\frac{L_N}{L_0}, \quad \log_2\frac{L_N}{L_0} = \log_2\frac{\delta_N}{\delta_0} = \log_2\frac{\Delta_N}{\Delta_0}.
$$
Ясно, что верно хотя бы одно из неравенств $L_N \leqslant 2L$, $\delta_N \leqslant 2\delta$ и $\Delta_N \leqslant 2\Delta$, что и обосновывает оценку \eqref{Oracle_Estim}.
\qed
\end{proof}
\begin{remark}
Оценка \eqref{Oracle_Estim} показывает, что в среднем трудоемкость итерации предложенного адаптивного алгоритма превышает трудоемкость аналогичного неадаптивного метода с постоянным шагом не более, чем в постоянное число раз. Отметим также, что при $k = 0,1,2,\ldots$ верно $L_{k+1} \leqslant 2CL$, где $C = \max\Big\{1, \displaystyle\frac{2\delta}{\delta_0},\displaystyle\frac{2\Delta}{\Delta_0}\Big\}$. Поэтому $S_N \leqslant \displaystyle\frac{N}{2CL}$, что указывает на скорость сходимости метода $O(\varepsilon^{-1})$, но при наличии в оценке \eqref{eq10} слагаемых, определяемых параметрами $\delta, \gamma, \Delta$ (при этом ввиду адаптивности метода $\delta_k$ и $\Delta_k$ могут быть меньше $\delta$ и $\Delta$ соответственно). Можно доказать, что эта величина ограничена в случае ограниченного допустимого множества задачи $Q$, что вполне может считаться оптимальным \cite{Polyak_1983}.
\end{remark}
\begin{remark}
Если дополнительно предположить, что в произвольной точке $x \in Q$ верно $f_{\delta}(x) = f(x)$, то в оценке \eqref{eq10} можно считать $\delta = 0$. В таком случае оценка \eqref{eq10} полностью адаптивна по параметрам $L, \Delta$ и $\delta$.
\end{remark}
\begin{remark}
Отметим, что ввиду адаптивности алгоритма 1  полученная в теореме~\ref{th01} оценка скорости сходимости может быть применена даже в случаях $L = + \infty$ или $\Delta = + \infty$. Если не происходит зацикливания и каждый раз выполняется критерий выхода из итерации, то алгоритм 1 применим и в этом случае. Пример, когда такое возможно ($\Delta = + \infty$), приведен в следующем разделе (задача 2).
\end{remark}

\section{О применимости метода к одному классу негладких задач за счет введения искусственных неточностей}\setcounter{equation}{0}

Отметим, что на величину $\Delta$ в \eqref{eq07} можно смотреть как на искусственную неточность, описывающую степень негладкости функционала $f$. Точнее говоря, $\Delta$ можно понимать, например, как верхнюю оценку суммы диаметров субдифференциалов $f$ в точках негладкости вдоль всевозможных векторных отрезков $[x; y]$ из области определения $f$. В \cite{Ston_new} введен следующий класс негладких выпуклых функционалов.
\begin{definition}\label{Main_Def}
Будем говорить, что выпуклый функционал $f\colon Q\rightarrow\mathbb{R}\;(Q\subset\mathbb{R}^n)$ имеет $(\delta,L)$-липшицев субградиент ($f\in C_{L,\hat{\Delta}}^{1,1}(Q)$), если для некоторых $\delta>0$ и $L >0 $

{\rm (i)} для произвольных $x,y\in Q$ выпуклый функционал $f$ дифференцируем во всех точках множества $\{y_t\}_{0\leq t\leq1}$ ($y_t = (1-t)x+ty$), за исключением последовательности (возможно, конечной)
\begin{equation}\label{eq4}
\{y_{t_j}\}_{j=1}^{\infty}:\;t_1<t_2<t_3<\ldots\text{ и }\lim_{j\rightarrow\infty}t_j=1;
\end{equation}

{\rm (ii)} для последовательности точек из \eqref{eq4} существуют конечные субдифференциалы в смысле выпуклого анализа $\{\partial f(y_{t_j})\}_{j=1}^{\infty}$ и
\begin{equation}\label{eq5}
{\rm diam}\;\partial f(y_{t_j})=:\hat{\Delta}_j>0,\ \ \text{ где }\sum_{j=1}^{+\infty}\hat{\Delta}_j=:\hat{\Delta}<+\infty;
\end{equation}

 {\rm (iii)} для произвольных $x,y\in Q$ при условии, что $y_t\in Q\setminus Q_0$ при всяком $t\in(0,1)$ (то есть существует градиент $\nabla f(y_k)$) для некоторой фиксированной константы $L>0$, не зависящей от выбора $x$ и $y$, выполняется неравенство
\begin{equation}\label{eq6}
\min_{\substack{\nabla f(x)\in\partial f(x),\; \nabla f(y)\in \partial f(y)}}\|\nabla f(x)- \nabla f(y)\|_*\leqslant L\|x-y\|.
\end{equation}
\end{definition}

В \cite{Ston_new}, в частности, показано, что всякий функционал $f \in C_{L,\delta}^{1,1}(Q)$ удовлетворяет для произвольного субградиента $\nabla f(x) \in \partial f(x)$ неравенству
\begin{equation}\label{IneqNonsmooth}
  f(y) \leqslant f(x) + \langle \nabla f(x), y - x \rangle + \frac{L}{2}\|y - x\|^2 + 2\hat{\Delta} \|y - x\| \quad \forall y \in Q.
\end{equation}

C другой стороны, ввиду выпуклости $f$ будет верно $f(y) \geqslant f(x) + \langle \nabla f(x), y - x \rangle$.
Поэтому всякая функция $f\in C_{L,\hat{\Delta}}^{1,1}(Q)$ удовлетворяет определению \ref{def1} при $\psi(y, x) = \langle \nabla f(x), y - x \rangle$ с параметрами $\delta = \gamma = 0$ и $\Delta = 2\hat{\Delta}$.

Оказывается, что экспериментально можно получать существенно лучшую скорость сходимости метода, чем по отмеченной выше оценке. Приведем некоторые примеры. Начнем с примера бесконечного числа точек негладкости (недифференцируемости) целевого функционала, но с конечной величиной $\Delta$.
\begin{task}[{\rm Аналог задачи Ферма --- Торричелли --- Штейнера}]\label{task2}
Для заданных шаров~$\omega_k$ с центрами $a_k=(a_{1k},a_{2k},\ldots,a_{nk})$ (координаты точек $a_k$ выбираются случайно так, чтобы $1<\sqrt{a_{1k}^2+a_{2k}^2+\ldots+a_{nk}^2}<1.5$, $k=\overline{1,m}$, $m=10$) и единичными радиусами в $n$-мерном евклидовом пространстве $\mathbb{R}^n$ ($n=10^5$) необходимо найти такую точку $x=(x_1,x_2,\ldots,x_n)$, чтобы целевая функция
$f(x):=\sum_{k=1}^m d(x,\omega_k) $ принимала наименьшее значение на множестве точек единичного шара с центром в нуле,
где $d(x,\omega_k)=\|x - a_k\|-1$, если $\|x - a_k\|>1$, и $0$  в противном случае (здесь $\|x - a_k\|=\sqrt{(x_1-a_{1k})^2+(x_2-a_{2k})^2+\ldots+(x_n-a_{nk})^2}$).
\end{task}

В таблице выше  для задачи 1 приведены усредненные результаты 10 экспериментов со случайным подбором координат точек для указанного количества итераций. Как видим, скорость сходимости метода близка к $O(\varepsilon^{-1})$. Это свойственно для неускоренных градиентных методов
на классе задач оптимизации выпуклых функций с липшицевым градиентом (так называемых гладких задач). Однако рассматриваемая задача негладкая, поскольку точки недифференцируемости~$f$ лежат в области определения (в единичном шаре с центром в нуле). Для задач минимизации выпуклых липшицевых функций, как известно, оптимальная оценка скорости сходимости (суб)градиентных методов --- $O(\varepsilon^{-2})$ \cite{NemYud_1979}. Оценку $O(\varepsilon^{-1})$ можно объяснить адаптивностью предложенного метода.

Рассмотрим еще пример, где довольно много точек негладкости. В частности, все точки некоторого векторного отрезка могут быть точками негладкости, и условие \eqref{IneqNonsmooth} выполнено лишь для бесконечного значения $\Delta$.
\begin{task}[\rm{Задача о наименьшем покрывающем шаре}]\label{task3}
Для заданных точек $$a_k=(a_{1k},a_{2k},\ldots,a_{nk})$$ найти евклидов шар наименьшего радиуса, в котором лежат эти точки. Координаты точек $a_k$ выбираются случайно так, что $0.5<\sqrt{a_{1k}^2+a_{2k}^2+\ldots+a_{nk}^2}<1,\;k=\overline{1,10},$ в $n$-мерном евклидовом пространстве $\mathbb{R}^n$ (размерность $n=10^5$) необходимо найти такую точку $x=(x_1,x_2,\ldots,x_n)$, чтобы целевая функция
$$f(x):=\max\limits_{k=\overline{1,m}} \|x - a_k\|=\max\limits_{k=\overline{1,m}}\sqrt{(x_1-a_{1k})^2+(x_2-a_{2k})^2+\ldots+(x_n-a_{nk})^2}$$
принимала наименьшее значение. Мы рассматриваем задачу на единичном шаре с центром в нуле.
\end{task}

В таблице выше для задачи 2 приведены усредненные результаты 10 экспериментов со случайным подбором координат точек для определенного количества итераций. Как видим, скорость сходимости метода снова близка к $O(\varepsilon^{-1})$.
\begin{table}
\centering
\small
\caption{Результаты численных экспериментов.}

\begin{tabular}{|c|c|c|c|c|c|c|c|c|c|c|}
\hline
         & \multicolumn{5}{c|}{Задача \ref{task2}}   & \multicolumn{5}{c|}{Задача \ref{task3}} \\ \hline
Итерации & 200    & 400    & 600    & 800   & 1000   & 200    & 400    & 600   & 800   & 1000  \\ \hline
Оценка   & 0.0232 & 0.0117 & 0.0079 & 0.006 & 0.0048 & 0.79   & 0.44   & 0.31  & 0.24  & 0.2   \\ \hline
Время, с & 27     & 54     & 82     & 110   & 136    & 15     & 29     & 44    & 58    & 72    \\ \hline
\end{tabular}
\end{table}

Приведенные результаты экспериментов указывают на потенциально неплохую эффективность предложенной адаптивной процедуры регулировки шага в методе. Отметим, что все вычисления были произведены с помощью программного обеспечения CPython 3.7 на компьютере с 3-ядерным процессором AMD Athlon II X3 450 с тактовой частотой 3,2 ГГц на каждое ядро. ОЗУ компьютера составляло 8 Гб.

Однако можно в некотором смысле и теоретически показать оптимальность предложенной схемы для рассматриваемых негладких задач. Оказывается, в случае известной величины \mbox{$\Delta< + \infty$} возможно несколько модифицировать алгоритм 1, обеспечив уменьшение\linebreak $\Delta_k\|x^{k+1}-x^k\|$ в \eqref{eq10} до любой заданной величины. Это позволит показать оптимальность данного метода в теории нижних оракульных оценок \cite{NemYud_1979} с точностью до логарифмического множителя.

Покажем, как это возможно сделать. Предположим, что на (k+1)-й итерации алгоритма~1 ($k=0,1,\ldots,N-1$) верно неравенство $L\leqslant L_{k+1}\leqslant2L$ (как показано в п.~2 доказательства теоремы \ref{th01}, этого можно всегда добиться выполнением не более чем постоянного числа операций п.~2 листинга алгоритма 1). Для каждой итерации алгоритма 1 ($k=0,1,\ldots,N-1$) предложим такую процедуру
\begin{equation}\label{eq12}
\fbox{\begin{minipage}{25em}
Повторяем операции п.~2 $p$ раз, увеличивая $L_{k+1}$ в два раза при неизменной $\Delta_{k+1}\leqslant2\Delta$.
\end{minipage}}
\end{equation}

Процедуру \eqref{eq12} остановим в случае выполнения одного из неравенств
\begin{equation}\label{eq13}
\Delta_{k+1}\big\|x^{k+1}-x^k\big\|\leqslant\frac{\varepsilon}{2}
\end{equation}
или
\begin{equation}\label{eq14}
 f(x^{k+1})\leqslant f(x^k)+\big\langle\widetilde{\nabla}f(x^k),x^{k+1}-x^k\big\rangle +2^{p-1}L\big\|x^{k+1}-x^k\big\|^2.
\end{equation}

Отметим, что здесь мы полагаем $f$ точно заданной, т.\,е. $f_{\delta}=f$ ($\delta=0$); $\widetilde{\nabla}f$~--- некоторый субградиент $f$. Процедура \eqref{eq12} предполагает на ($k+1$)-й итерации ($k=0,1,2,\ldots,N-1$) обновления $x^{k+1}$ (при сохранении $x^k$). Оценим количество повторений $p$ шага п.~2 листинга алгоритма 1, необходимое для достижения альтернативы \eqref{eq13}, \eqref{eq14}. Для всяких $x^k, x^{k+1}\in Q$ по предположению верно неравенство
$$f(x^{k+1})\leqslant f(x^k)+\big\langle\widetilde{\nabla}f(x^k),x^{k+1}-x^k\big\rangle+\frac{L}{2}\big\|x^{k+1}-x^k\big\|^2+\Delta\big\|x^{k+1}-x^k\big\|,$$
причем $\Delta_{k+1}\leqslant2\Delta$. Если не выполнено \eqref{eq13}, то $\big\|x^{k+1}-x^k\big\|>\displaystyle\frac{\varepsilon}{4\Delta}$ и \eqref{eq14} заведомо верно при
\begin{equation}\label{eq15}
2^p>1+\frac{16\Delta^2}{\varepsilon L},
\end{equation}
поскольку тогда ввиду \eqref{eq13}
$$\frac{2^p-1}{2}L\big\|x^{k+1}-x^k\big\|^2>2\Delta\big\|x^{k+1}-x^k\big\|.$$

Итак, после повторения $p$ процедур ($p$ удовлетворяет \eqref{eq15}) типа \eqref{eq12} на каждой из $N$ итераций алгоритма 1 неравенство \eqref{eq10} примет вид
$$f(\hat{x})-f^*\leqslant\frac{R^2}{S_N}+\frac{\varepsilon}{2},\  \text{ где }\  S_N=\sum\limits_{k=0}^{N-1}\frac{1}{L_{k+1}}\geqslant\frac{N}{2^{p+1}L}.$$
Поэтому
$\displaystyle\frac{R^2}{S_N}\leqslant\displaystyle\frac{2^{p+1}LR^2}{N}\leqslant\displaystyle\frac{\varepsilon}{2}$
в случае $N\geqslant\displaystyle\frac{2^{p+1}LR^2}{2}$. С учетом \eqref{eq15} получаем оценку
\begin{equation}\label{eq16}
N\geqslant\frac{4LR^2}{\varepsilon}+\frac{64\Delta^2R^2}{\varepsilon^2}.
\end{equation}
При этом \eqref{eq15} означает, что на каждой итерации потребуется не более чем $p = \log_2\Big(1+\displaystyle\frac{16\Delta^2}{\varepsilon L}\Big)$ шагов типа п.~2 листинга алгоритма 1  (т.\,е. операций проектирования) для стандартной модели $\psi(y,x) = \langle \nabla f(x), y -x \rangle $. Итак, верна
\indent\begin{theorem}\label{thm2ston}
Пусть функция $f$ удовлетворяет определению $\ref{def1}$ при $\psi(y, x) = \langle \nabla f(x), y - x \rangle$ с параметрами $\delta = \gamma = 0$ и $\Delta > 0$. Тогда в обозначениях теоремы $\ref{th01}$ для выхода $\hat{x}$ модифицированного алгоритма $1$ c учетом дополнительной процедуры $\eqref{eq12}$ неравенство $f(\hat{x})-f^*\leqslant\varepsilon$ будет гарантированно выполнено не более чем после
\begin{equation}\label{eq16fin}
\cebr{\Big(\frac{4LR^2}{\varepsilon}+\frac{64\Delta^2R^2}{\varepsilon^2}\Big)\log_2 \Big(1+\frac{16\Delta^2}{\varepsilon L} \Big)}
\end{equation}
вычислений субградиента $f$.
\end{theorem}

Таким образом доказано, что для рассмотренного класса негладких задач приемлемое качество решения можно достичь за $O\br{\varepsilon^{-2}\log_2 \varepsilon^{-1}}$ вычислений субградиента $f$, что близко к оптимальной оценке c точностью до логарифмического множителя. Отметим, что примеры сходимости метода со скоростью $O(\varepsilon^{-1})$ для некоторых выпуклых негладких задач наблюдались и для так называемого универсального градиентного метода \cite{Nesterov_2015} с другой концепцией искусственной неточности. Однако для негладких задач с липшицевым целевым функционалом в \cite{Nesterov_2015} доказана оценка скорости сходимости вида $O\br{M_f \varepsilon^{-2}}$, зависящая еще от константы Липшица целевого функционала $M_f$. Полученная нами оценка \eqref{eq16fin} может быть лучше при малом $\Delta>0$ (в этом случае оценка \eqref{eq16fin} близка к $O\br{\varepsilon^{-1}}$).

\section{Метод для минимизации функций, удовлетворяющих условию градиентного доминирования при неточном задании целевой функции и градиента}\setcounter{equation}{0}

Теперь предложим подход к задаче минимизации, вообще говоря, невыпуклых функций с неточно заданным градиентом. При этом метод предполагает адаптивную настройку на некоторые параметры, в том числе связанные с величиной погрешности задания градиента. Пусть рассматривается задача минимизации функции на всем пространстве $f:\mathbb{R}^n\rightarrow\mathbb{R}$, для которой (всюду далее полагаем норму $\|\cdot\|$ евклидовой)
\begin{itemize}
\item[(i)] существует $x_*\in \mathbb{R}^n$ такое, что
\begin{equation}\label{eq17}
f(x_*)=\min\limits_{x\in \mathbb{R}^n} f(x)=:f^*;
\end{equation}
\item[(ii)] выполнено условие Поляка --- Лоясиевича \eqref{eq18} (или $(PL)$-условие);
\item[(iii)] для некоторых постоянных $L>0$ и $\Delta >0$ верно неравенство
\begin{equation}\label{eq21}
f(y)\leqslant f(x)+\abr{\nabla f(x),y-x}+\frac{L\nbr{y-x}^2}{2}\ \ \forall x,y\in \mathbb{R}^n.
\end{equation}
\end{itemize}

Если предположить, что в каждой точке $x\in \mathbb{R}^n$ доступно приближенное значение $\widetilde{\nabla}f(x)$ градиента $\nabla f(x)\colon \big\|\widetilde{\nabla}f(x)-\nabla f(x)\big\|\leqslant\Delta\;\;\forall x\in \mathbb{R}^n
$ при некотором фиксированном $\Delta>0$, то \eqref{eq21} верно c  заменой градиента $\nabla f(x)$ на $\widetilde{\nabla}f(x)$.
Далее для удобства будем обозначать $g_x:=\nbr{\nabla f(x)}$ и $\widetilde{g}_x:=\|\widetilde{\nabla}f(x)\|$. К задаче \eqref{eq17} будем применять градиентный метод вида
\begin{equation}\label{eq22}
x^{k+1}=x^k-h_k\widetilde{\nabla}f(x^k),
\end{equation}
$k=0,1,2,\ldots$ и $h_k>0$. При этом $h_k$ выберем так, чтобы
\begin{equation}\label{eq23}
f(x^{k+1})\leqslant f(x^k)+\big\langle\widetilde{\nabla}f(x^k),x^{k+1}-x^k\big\rangle+\frac{L\big\|x^{k+1}-x^k\big\|^2}{2}+\Delta_{k+1}\big\|x^{k+1}-x^k\big\|,
\end{equation}
где $\Delta_{k+1}>0$~--- адаптивно подбираемая величина. В начале каждой итерации $\Delta_{k+1}:=\frac{\Delta_k}{2}$, а далее $\Delta_{k+1}$ ($k=0,1,2,\ldots$) увеличивается в два раза, и процедура \eqref{eq22} повторяется до тех пор, пока не выполняется \eqref{eq23}.

Ясно, что \eqref{eq23} заведомо верно при $\Delta_{k+1}\geqslant\Delta$. Поэтому аналогично п.~2) доказательства теоремы \ref{th01} проверяется, что за конечное число таких шагов \eqref{eq23} будет выполнено на любой итерации ($k=0,1,2,\ldots$), после чего $$f(x^{k+1})-f(x^k)\leqslant\varphi(h_k), \text{  где  } \varphi(h)=-h\widetilde{g}_{x^k}^2+\displaystyle\frac{Lh^2}{2}\widetilde{g}_{x^k}^2+\Delta_{k+1}h\widetilde{g}_{x^k}.
$$
Выберем шаг $h_k$ так, чтобы минимизировать величину $\varphi(h_k)$, т.\,е. $\varphi'(h_k)=0$, и
\begin{equation}\label{eq24}
h_k=\frac{1}{L}-\frac{\Delta_{k+1}}{L\widetilde{g}_{x^k}}.
\end{equation}

В таком случае \eqref{eq23} означает, что
\begin{equation}\label{eq25}
f(x^{k+1})-f(x^k)\leqslant-\frac{1}{2L}\br{\widetilde{g}_{x^k}-\Delta_{k+1}}^2\leqslant-\frac{1}{2L}\br{\frac{\widetilde{g}_{x^k}-\Delta_{k+1}}{\widetilde{g}_{x^k}+\Delta}}g_{x^k}^2,
\end{equation}
поскольку
$\mbr{\widetilde{g}_{x^k} - g_{x^k}}\leqslant\big\|\widetilde{\nabla}f(x^k)-\nabla f(x^k)\big\|\leqslant\Delta$ и $\widetilde{g}_{x^k}+\Delta\geqslant g_{x^k}$. Неравенство \eqref{eq25} означает, что для произвольного $k=0,1,2,\ldots$
$$f(x^k)-f(x^{k+1})\geqslant\frac{1}{2L}\br{\frac{\widetilde{g}_{x^k}-\Delta_{k+1}}{\widetilde{g}_{x^k}+\Delta}}^2g_{x^k}^2\geqslant
\frac{\mu}{L}\br{\frac{\widetilde{g}_{x^k}-\Delta_{k+1}}{\widetilde{g}_{x^k}+\Delta}}^2\br{f(x^k)-f^*}$$
ввиду $(PL)$-условия \eqref{eq18}. Поэтому
$$f(x^{k+1})-f(x_*)\leqslant\br{1-\frac{\mu}{L}\br{\frac{\widetilde{g}_{x^k}-\Delta_{k+1}}{\widetilde{g}_{x^k}+\Delta}}^2}\br{f(x^k)-f(x_*)},$$
откуда
\begin{equation}\label{eq26}
f(x^{k+1})-f^*\leqslant\prod_{i=0}^k\br{1-\frac{\mu}{L}\br{\frac{\widetilde{g}_{x^i}-\Delta_{i+1}}{\widetilde{g}_{x^i}+\Delta}}^2}\br{f(x^0)-f^*}.
\end{equation}
Можно считать, что $\mu\leqslant L$, и ввиду $\widetilde{g}_{x^i}-\Delta_{i+1}<\widetilde{g}_{x^i}+\Delta$ в \eqref{eq26} справа входит произведение $k+1$ числа, каждое из которых меньше 1. Адаптивность подбора $\Delta_{k+1}\leqslant2\Delta$ на каждой итерации может привести к увеличению дроби $\displaystyle\frac{\widetilde{g}_{x^k}-\Delta_{k+1}}{\widetilde{g}_{x^k}+\Delta}$ и уменьшению множителей в \eqref{eq26}, что потенциально улучшает оценку по сравнению с неадаптивным вариантом
\begin{equation}\label{eq27}
f(x^{k+1})-f^*\leqslant\prod_{i=0}^k\br{1-\frac{\mu}{L}\br{\frac{\widetilde{g}_{x^i}-\Delta}{\widetilde{g}_{x^i}+\Delta}}^2}\br{f(x^0)-f^*}.
\end{equation}

Аналогичную оценку можно предложить, также используя следующий метод с адаптивным выбором не только погрешности на итерациях, но и величины $L$ (см. алгоритм~2).

\begin{algorithm}
\begin{algorithmic}\small
\caption{Адаптивный градиентный метод для функций, удовлетворяющих (PL)-условию.}
\REQUIRE $x^0$~--- начальная точка, параметры $\Delta_0,\;L_0$\\
($2\mu \leqslant L_0 < 2L,\;\Delta_0\leqslant2\Delta$).
\STATE $L_{k+1}:= \nicefrac{L_k}{2},\;\Delta_{k+1}:=\nicefrac{\Delta_k}{2}$.
\STATE $x^{k+1}=x^k-h_k\widetilde{\nabla}f(x^k),$
$$ h_k=\frac{1}{L_{k+1}}-\frac{\Delta_{k+1}}{L_{k+1}\widetilde{g}_{x^k}},\;\widetilde{g}_{x^k}=\big\|\widetilde{\nabla}f(x^k)\big\|.$$
\REPEAT
\IF{$f(x^{k+1})\leqslant f(x^k)+\abr{\widetilde{\nabla}f(x^k),x^{k+1}-x^k}+\displaystyle \frac{L_{k+1}}{2}\big\|x^{k+1}-x^k\big\|^2+\Delta_{k+1}\big\|x^{k+1}-x^k\big\|$}
\STATE $k:=k+1$ и выполнение п.~1.
\ELSE
\STATE $L_{k+1}:=2\cdot L_{k+1};\;\Delta_{k+1}:=2\cdot\Delta_{k+1}$ и выполнение п.~2.
\ENDIF
\UNTIL{$k>=N$}
\ENSURE $x^{k+1}$.
\end{algorithmic}
\end{algorithm}

Для данного алгоритма будет верна оценка \eqref{eq27new}, обоснование которой аналогично \eqref{eq27}. Более того, можно показать, что либо невязка $\min\limits_kf(x^k)-f^*$ убывает со скоростью геометрической прогрессии при увеличении $k$ (см.~\eqref{eq28}), либо она ограничена величиной $\Delta$ (см.~\eqref{eq29}). Справедливо следующее утверждение.\pagebreak
\begin{theorem}\label{thm3ston}
После $k$ итераций алгоритма $2$ будет выполняться следующее неравенство:
\begin{equation}\label{eq27new}
f(x^{k+1})-f^*\leqslant\prod_{i=0}^k\br{1-\frac{\mu}{L_{k+1}}\br{\frac{\widetilde{g}_{x^i}-\Delta_{i+1}}{\widetilde{g}_{x^i}+\Delta}}^2}\br{f(x^0)-f^*}.
\end{equation}
Более того, если дополнительно потребовать для алгоритма $2$ $\Delta_{k+1}=\min\{\Delta_{k+1},\Delta\}$, то для всякого $C>1$ будет выполняться одно из двух неравенств
\begin{equation}\label{eq28}
f(x^{k+1})-f^*\leqslant\Big(1-\frac{\mu}{L}\Big(\frac{C-1}{C+1}\Big)^2\Big)^{k+1}\br{f(x^0)-f^*}
\end{equation}
или
\begin{equation}\label{eq29}
\min\limits_{i=\overline{1,k+1}}f(x^i)-f^*<\frac{(C+1)^2\Delta^2}{2\mu}.
\end{equation}
\end{theorem}
\begin{proof}
Отметим лишь, что при произвольном $ k \geqslant 0$ верно
$$\frac{\widetilde{g}_{x^k}-\Delta_{k+1}}{\widetilde{g}_{x^k}+\Delta}\geqslant\frac{\widetilde{g}_{x^k}-\Delta}{\widetilde{g}_{x^k}+\Delta}=1-\frac{2\Delta}{\widetilde{g}_{x^k}+\Delta}.$$
Пусть $\widetilde{g}_{x^k}\geqslant C\Delta$ для некоторой постоянной $C>1$. Тогда $1-\displaystyle\frac{2\Delta}{\widetilde{g}_{x^k}+\Delta}\geqslant1-\displaystyle\frac{2}{C+1}=\displaystyle\frac{C-1}{C+1}>0$ и \eqref{eq26} принимает вид \eqref{eq28}. Если же для некоторого $k$ верно $\widetilde{g}_{x^k}<C\Delta$, то $g_{x^k}<C\Delta+\Delta=\Delta(C+1)$, и \eqref{eq29} верно в силу $(PL)$-условия \eqref{eq18}.
\qed
\end{proof}
\begin{remark}\label{PLRemInexact}
Можно рассматривать вместо \eqref{eq22} более слабое условие
\begin{equation}\label{SS1}
f_{\delta}(y)\leqslant f_{\delta}(x)+\langle\widetilde{\nabla}f(x),y-x\rangle + \frac{L}{2}\|y-x\|^{2}+\Delta\|y-x\|+\delta\quad \forall\,x,y\in Q
\end{equation}
для некоторого $\delta>0$ и приближения $f_{\delta}$: $f_{\delta} (x) \leq f(x) \leq f_{\delta} + \delta$. Например, это актуально в случае, если значения $f$ немного отличаются от значений некоторой достаточно гладкой функции $\widetilde{f}$, удовлетворяющей \eqref{eq22} (при этом $\widetilde{\nabla}f(x)$ --- некоторое возмущенное с точностью $\Delta$ значение градиента $\nabla\widetilde{f}(x)$). Тогда рассмотрим метод \eqref{eq23}--\eqref{eq25} с видоизмененным критерием выхода из итерации \eqref{eq24}
\begin{equation}\label{SS1}
f_{\delta}(x^{k+1})\leqslant f_{\delta}(x^{k})+\langle\widetilde{\nabla}f(x^{k}),x^{k+1}-x^{k}\rangle + \frac{L_{k+1}\|x^{k+1}-x^{k}\|^{2}}{2}+\Delta_{k+1}\|x^{k+1}-x^{k}\|+\delta_{k+1},
\end{equation}
который предполагает адаптивный подбор величин $\Delta_{k+1}$ и $\delta_{k+1}$ при заданных изначально $L_{0}\leqslant 2L$, $\Delta_{0}\leqslant\Delta$ и $\delta_{0}\leqslant2\delta$. Тогда на каждой итерации вместо неравенства \eqref{eq25} будет верно
$$f(x^{k})-f(x^{k+1})+\delta_{k+1}+\delta \geqslant f_{\delta}(x^{k})-f_{\delta}(x^{k+1})+\delta_{k+1}\geqslant\frac{\mu}{L_{k+1}}\Big(\frac{\widetilde{g}_{x^{k}}-\Delta_{k+1}}{\widetilde{g}_{x^{k}}+\Delta}\Big)^{2}\big(f(x^{k})-f^{*}\big),$$
откуда аналогично \eqref{eq26} имеем
$$f(x^{k+1})-f^{*}\leqslant\Big(1-\frac{\mu}{L_{k+1}}\Big(\frac{\widetilde{g}_{x^{k}}-\Delta_{k+1}}{\widetilde{g}_{x^{k}}+\Delta}\Big)^{2}\Big)\big(f(x^{k})-f^{*}\big)+\delta_{k+1}+\delta$$
$$\leqslant\Big(1-\frac{\mu}{L_{k+1}}\Big(\frac{\widetilde{g}_{x^{k}}-\Delta_{k+1}}{\widetilde{g}_{x^{k}}+\Delta}\Big)^{2}\Big)\Big(1-\frac{\mu}{L_{k}}\Big(\frac{\widetilde{g}_{x^{k-1}}-\Delta_{k}}{\widetilde{g}_{x^{k-1}}+\Delta}\Big)^{2}\Big)\big(f(x^{k-1})-f^{*}\big)$$
$${ } +\  (\delta_{k}+\delta)\Big(1-\frac{\mu}{L_{k+1}}\Big(\frac{\widetilde{g}_{x^{k-1}}-\Delta_{k}}{\widetilde{g}_{x^{k-1}}+\Delta}\Big)^{2}\Big)+\delta+\delta_{k+1}\leqslant\ldots$$
$$\leqslant\prod^{k}_{i=0}\Big(1-\frac{\mu}{L_{i+1}}\Big(\frac{\widetilde{g}_{x^{i}}-\Delta_{i+1}}{\widetilde{g}_{x^{i}}+\Delta}\Big)^{2}\Big)\big(f(x^{0})-f^{*}\big)+$$
$$+\sum_{i=0}^{k-1}(\delta+\delta_{i+1})\prod^{k}_{j=i}\Big(1-\frac{\mu}{L_{j+1}}\Big(\frac{\widetilde{g}_{x^{j}}-\Delta_{j+1}}{\widetilde{g}_{x^{j}}+\Delta}\Big)^{2}\Big)+\delta_{k+1}+\delta.$$

Полученная оценка выглядит несколько громоздко. Конкретизируя ее при постоянном $L_{i+1} = L$, $\Delta=0$ и $\delta_{i+1}\leqslant2\delta$ ($i\geqslant0$), получаем
$$f(x^{k+1})-f^{*}\leqslant\Big(1-\frac{\mu}{L}\Big)^{k+1}\big(f(x^{0})-f^{*}\big)+\sum_{i=0}^{k}(\delta+\delta_{i+1})\Big(1-\frac{\mu}{L}\Big)^{k-i}$$
$$\leqslant\Big(1-\frac{\mu}{L}\Big)^{k+1}\big(f(x^{0})-f^{*}\big)+3\delta\sum_{i=0}^{k}\Big(1-\frac{\mu}{L}\Big)^{k-i}=\Big(1-\frac{\mu}{L}\Big)^{k+1}\big(f(x^{0})-f^{*}\big)+\frac{3\delta L}{\mu}.$$

Данное неравенство приводит к таким выводам. С одной стороны мы видим, что величина, связанная с погрешностью $\delta$, ограничена. Однако она может быть довольно немалой при большом значении числа обусловленности $\displaystyle\frac{L}{\mu}$. Это показывает также, что замена слагаемого $\Delta\|y-x\|$ в \eqref{SS1} на $\displaystyle\frac{\Delta^2 + \|y-x\|^2}{2}$ может привести к ухудшению оценки качества решения при достаточно большом $\displaystyle\frac{L}{\mu}$.
\end{remark}
\begin{remark}
Аналогично второй части доказательства теоремы \ref{th01} можно проверить, что трудоемкость итерации адаптивного алгоритма 2 сопоставима с трудоемкостью аналогичного неадаптивного метода.
\end{remark}
\begin{remark}
Предложенные в этом разделе подходы можно применять и для некоторых задач негладкой оптимизации (см. предыдущий раздел и определение \ref{Main_Def}), для которых целевая функция удовлетворяет (PL)-условию (в частности, $\mu$-сильно выпукла).
Если определяющий степень негладкости функции параметр $\Delta$ достаточно мал, то найденные оценки \eqref{eq26}--\eqref{eq29} (см. также замечание \ref{PLRemInexact}) позволяют сделать вывод о близкой к линейной скорости сходимости.
\end{remark}

\section*{Заключение}

В настоящей работе рассмотрены некоторые подходы к концепции неточной модели целевой функции в оптимизации, которые учитывают как погрешность задания целевого функционала, так и погрешность задания градиента. Предложены методы с адаптивным выбором шага, а также адаптивной настройкой величины в оценке скорости сходимости, которая определяется упомянутыми погрешностями.

Cравнивая алгоритмы 1 и 2, между отметим следующее. Преимущества алгоритма 1 (и его модификации из третьего раздела статьи) состоят в максимальной общности (метод можно использовать для широкого класса задач выпуклой оптимизации \cite{s2, inexact_model_2019}, в том числе с условиями относительной гладкости \cite{s3}). Также, в отличие от алгоритма 2, для работы алгоритма 1 и использования найденной оценки скорости сходимости нет необходимости знать $\Delta$ (оценку неточности задания градиента). Как преимущества алгоритма 2 для функций, удовлетворяющих ($PL$)-условию, можно упомянуть близкую к линейной скорость сходимости и возможность использования метода на неограниченном допустимом множестве. Однако для оценок \eqref{eq26}--\eqref{eq29} необходимо знать верхнюю оценку величины $\Delta$. Также существенно использована безусловность поставленной задачи. Оценка для алгоритма 1 в свою очередь проигрывает алгоритму 2 возможностью сколь угодно большого влияния погрешности градиента при $\gamma>0$ для неограниченной области $Q$. Хорошо известно, что ($PL$)-условие заведомо верно в случае $\mu$-сильной выпуклости целевой функции $f$. Однако довольно хорошо известны примеры, когда нельзя быть уверенным даже в выпуклости $f(x)$, но ($PL$)-условие имеет место (см., например, разд.~4.3 из диссертации \cite{Nesterov_2013}). Это означает, что алгоритм 2 применим и для некоторых задач невыпуклой оптимизации. Интересно, что все рассмотренные методы применимы к некоторому классу задач негладкой оптимизации (см. определение \ref{Main_Def}).

В качестве актуальной задачи на будущее можно было бы выделить проблему построения так называемых ускоренных методов для рассмотренных классов задач. В частности, к ускоренным методам относят самые разные вариации так называемого быстрого градиентного метода (БГМ) (см., например, \cite{Necoara_2019, s1, Nesterov_2013}). Для задач выпуклой гладкой оптимизации без погрешностей БГМ гарантирует лучшую оценку скорости сходимости по сравнению с \eqref{eq02}. Известно также, что в сильно выпуклом случае использование ускоренных методов позволяет уменьшить знаменатель геометрической прогрессии, которая описывает скорость сходимости. Более того, неадаптивные ускоренные методы для релаксаций условия сильной выпуклости исследовались в \cite{Necoara_2019}. Однако стоит отметить, что при наличии погрешностей ситуация становится уже менее тривиальной: в отличие от обычного градиентного метода возможно их накопление в итоговой оценке \cite{DevolderThesis}, либо же необходимо использовать довольно ограничительные условия на величины таких погрешностей \cite{Artemont}. Также пока не удалось предложить ускоренный метод, который применим в общем случае для относительно гладких задач \cite{s3}. Представляется интересной задача исследования применимости результатов настоящей работы для приближенного решения бесконечномерных задач, в частности, для некоторых типов линейных и нелинейных операторных уравнений.\smallskip

Автор благодарит Александра Владимировича Гасникова, а также рецензента за полезные обсуждения и замечания.

\end{document}